\documentclass[12pt]{article}
\usepackage[utf8]{inputenc}
\usepackage{amssymb}
\usepackage[margin=2.5cm]{geometry}
\usepackage{amsmath,amsthm,graphicx,float}
\usepackage{longtable}
\usepackage{xcolor}
\usepackage{cleveref}
\usepackage{subcaption}
\usepackage{enumitem}
\usepackage{dsfont}
\usepackage{caption}
\usepackage{verbatim}

\usepackage{tikz}
\usepackage{tikz-cd}
\usetikzlibrary{positioning}
\usetikzlibrary{decorations.pathreplacing}
\usetikzlibrary{graphs}

\tikzstyle{mynode} = [circle, draw, minimum size=0.1cm, inner sep=0pt]
\tikzstyle{mylist} = [inner sep=2, font=\footnotesize]
\tikzstyle{myedge} = [midway, above, inner sep=1mm, font=\footnotesize]
\tikzstyle{myconflictx} = [pos=0.1, inner sep=1mm, font=\footnotesize]
\tikzstyle{myconflicty} = [pos=0.9, inner sep=1mm, font=\footnotesize]

\newtheorem{theorem}{Theorem}[section]
\newtheorem{problem}[theorem]{Problem}

\newtheorem{lemma}[theorem]{Lemma}
\newtheorem{proposition}[theorem]{Proposition}
\newtheorem{corollary}[theorem]{Corollary}

\newcommand{\single}{\chi_{\nleftrightarrow}}
\newcommand{\ch}{\mathrm{ch}}
\newcommand{\chsep}{\mathrm{ch}_{sep}}
\newcommand{\chad}{\mathrm{ch}_{ad}}
\newcommand{\sing}{\chi_{\nleftrightarrow}}

\begin{document}
\title{Single conflict coloring, adaptable choosability and separation 
choosability\footnote{This paper is partially based on the thesis written 
by the second author under supervision of the first author}}

\author{
{\sl Carl Johan Casselgren}\thanks{{\it E-mail address:} 
carl.johan.casselgren@liu.se}\\ 
Department of Mathematics \\
Link\"oping University \\ 
SE-581 83 Link\"oping, Sweden
\and
{\sl Kalle Eriksson}\thanks{{\it E-mail address:} 
kalle.eriksson@math.su.se}\\ 
Department of Mathematics \\
Stockholm University \\ 
SE-106 91 Stockholm, Sweden
}

\maketitle


\begin{abstract}
We study relations between three interrelated notions 
of graph (list) coloring:  single conflict coloring, adapted list coloring and choosability
with separation (with $1$ overlapping color between 
lists of adjacent vertices), and 
their respective invariants single conflict chromatic number $\single$, 
adaptable choosability $\ch_{ad}$ and separation 
choosability $\ch_{sep}$.
We 
investigate graphs with small values of these invariants,
and construct explicit families of graphs $G$ with
$\single(G) = \ch_{ad}(G) > \ch_{sep}(G)$, as well as where all three invariants are equal.
Furthermore, we consider planar graphs and investigate for which 
triples $(a,b,c)$, there is a planar graph $G$ with 
$(\ch_{sep}(G), \ch_{ad}(G), \single(G)) = (a,b,c)$.
Throughout the paper we pose many questions
on these graph coloring parameters, and discuss connections to related coloring 
invariants such as 
adapted coloring.
\end{abstract}



\section{Introduction}
	  Given a graph $G$, assign to each vertex $v$ of $G$ a set
        $L(v)$ of colors.
	Such an assignment $L$ is called a 
    \emph{list assignment} for $G$ and
    the sets $L(v)$ are referred
    to as \emph{lists} or \emph{color lists}.
    If there is a proper coloring $\varphi$ of $G$ such that 
    $\varphi(v) \in L(v)$ for all $v \in V(G)$, then
	$G$ is \emph{$L$-colorable} and $\varphi$
    is called an \emph{$L$-coloring} or simply a {\em list coloring}. 
	Furthermore, $G$ is
    called \emph{$k$-choosable} if it is $L$-colorable
    for every list assignment $L$ where all lists have size at least $k$. 
	The minimum $k$ such that $G$ is $k$-choosable is the 
	{\em list-chromatic number} (or {\em choice number})  $\ch(G)$ of $G$.
	
Being first introduced in the 70's \cite{ERT, Vizing}, list coloring is a well-established area of graph coloring. Indeed, many graph coloring models  translate naturally to the list coloring setting, and numerous variants of list coloring are studied in the literature. In this paper we study relations between three different variants of list coloring.

Given an edge coloring $f$ of a graph $G$, a coloring $c : V(G) \to \mathbf{Z}$ is {\em adapted} to $f$ if no edge $xy$ satisfies
that $c(x) = c(y) = f(xy)$. Adapted coloring was introduced by
Hell and Zhu \cite{HellZhu}, who
among other things proved that a graph is adaptably $2$-colorable
if and only if it contains an edge whose removal yields a bipartite graph. They also proved general bounds on the 
{\em adaptable chromatic number} $\chi_a(G)$,
that is, the least number of colors needed for an adapted coloring of a graph $G$.

Building on the notion of adapted coloring, Kostochka and Zhu
\cite{KostochkaZhu} studied the list coloring version of adapted coloring, where the vertices are assigned lists, and we seek
a list coloring that is adapted to a given edge coloring. The minimum size of the lists which ensures the existence of such an
adapted list coloring is the {\em adaptable choice number}
$\ch_{ad}(G)$ (or {\em adaptable choosability}) of a graph $G$. Kostochka and Zhu proved a general upper bound on this invariant and also characterized graphs with adaptable choosability $2$.

Another variant of list coloring is so-called
choosability with separation, introduced by Kratochvil, Tuza and Voigt
\cite{KratochvilTuzaVoigt}. In this model, we require that color lists of adjacent vertices have bounded intersection,
and ask for the minimum size of the lists which guarantees existence of a proper list coloring of the graph in question.
We shall focus on the variant where lists of adjacent vertices have
at most one common color, often referred to as the {\em separation choosability}.
Denote by $\ch_{sep}(G)$ the smallest size of the lists which ensures that $G$ is colorable
from such lists, provided that lists of adjacent vertices intersect in at most one color. This
particular variant was specifically studied in e.g.\ \cite{EsperetKangThomasse, FurediKostochkaKumbhat}.

Generalizing both the notion of adaptable choosability and separation choosability,
single conflict coloring was introduced by Dvorak et al.\ \cite{DvorakEsperetKangOzeki}. 
A {\em local $k$-partition} of a graph $G$ is a collection $\{c_v\}_{v \in V(G)}$ of maps 
	$c_v : E(v) \to \{1,\dots, k\}$, where $E(v)$ denotes the set of edges incident
	with $v$. We think of $c$ as a conflict map, and for each
	edge $e = xy \in E(G)$, $c$ yields a conflict-pair 
	$(c_{x}(e), c_{y}(e))$. Given such a local $k$-partition $\{c_v\}$,
	$G$ is {\em conflict $\{c_v\}$-colorable} if there is a map 
	$\varphi: V(G) \to \{1,\dots, k\}$ such that no edge $e = xy$ 
	satisfies that $\varphi(x) = c_{x}(e)$ and $\varphi(y) = c_{y}(e)$.
	The {\em single conflict chromatic number} $\single(G)$ of
	$G$ is the smallest $k$ such that
	$G$ is conflict $\{c_v\}$-colorable for every local $k$-partition $\{c_v\}$.

As pointed out in \cite{DvorakEsperetKangOzeki}, separation choosability, 
adaptable choosability, and single conflict coloring are strongly related.
Indeed, we have
\begin{equation}
\label{eq:olika}
	\single(G) \geq \ch_{ad}(G) \geq \ch_{sep}(G),
\end{equation}
where the latter inequality has been observed in several places in the 
literature 	
\cite{EsperetKangThomasse, CasselgrenGranholmRaspaud, DvorakEsperetKangOzeki}.
Furthermore, the adaptable choosability is at most the ordinary 
choosability, so
we have that $$\ch(G) \geq \ch_{ad}(G) \geq \ch_{sep}(G).$$
Let us also note the following inequality for the single conflict chromatic number.

\begin{proposition}
\label{prop:chsingle}
	For every positive integer $k$, there are graphs $G_k$ and $H_k$
	such that $\single(G_k)- \ch(G_k) \geq k$, and $\ch(H_k) - \single(H_k) \geq k$.
\end{proposition}
	This proposition follows by taking $G_k$ to be 
         be the complete bipartite graph on $k+k$ vertices
	and $H_k$ the complete graph of order $k$, and applying results
	from \cite{DvorakEsperetKangOzeki, FurediKostochkaKumbhat}.
	In particular, this implies that the difference between
	$\single$ and $\ch_{ad}$ can be arbitrarily large.

Dvorak and Postle \cite{DvorakPostle} introduced correspondence coloring
(also known as DP-coloring) as a generalization of list coloring. 
Roughly, the idea is to replace every vertex in a graph $G$ 
by a $k$-clique, and adding a matching between any pair
of cliques corresponding to adjacent 
vertices. (The vertices in the $k$-clique correspond to the colors in a 
list of size $k$.)
If there is an independent set
in the resulting graph of size $|V(G)|$ for any choice of the matchings 
between	cliques corresponding to adjacent vertices in $G$, then the 
graph is {\em $k$-DP-colorable}.
The {\em DP-chromatic number $\chi_{DP}(G)$} is the minimum $k$ such 
that $G$ is $k$-DP-colorable\footnote{See \cite{DvorakPostle} for a 
formal definition.}.  Clearly, $\ch(G) \leq \chi_{DP}(G)$
for any graph $G$, since ordinary list coloring is the case when the 
matching between different cliques joins identical colors of the 
corresponding color lists.
This generalization of ordinary list coloring has been widely studied 
e.g.\ for planar graphs. 
In contrast to Proposition \ref{prop:chsingle}, the single conflict 
chromatic number is bounded by the DP-chromatic number.
Indeed, it is clear that $\single(G) \leq \chi_{DP}(G)$
for every graph $G$, although this is not explicitly mentioned in 
\cite{DvorakEsperetKangOzeki}.

A main aim of this note is to give a unified treatment and discuss similarities and differences between single conflict coloring, adaptable choosability, and separation choosability.
Throughout the paper
we point out that several results which were proved in the setting of adaptable
choosability or choosability with separation also hold in the setting of single conflict 
coloring. Along the way we also make some remarks about adapted coloring.

We characterize graphs with single conflict chromatic number $2$, and give a structural description of graphs
with separation choosability $2$. 
As it turns out, the graphs with single chromatic number
$2$ are precisely the graphs with adaptable choosability $2$, while the family with
separation choosability $2$ is much larger.
Using these results, we
construct a large family of graphs $\mathcal{G}$ such that
$\single(G) =\ch_{ad}(G) > \ch_{sep}(G)$ for every $G \in \mathcal{G}$. 
We also give an infinite family of graphs where all three invariants are equal.

Next, we consider planar graphs. Our investigation of the invariants $\ch_{sep}, \ch_{ad}$
and $\single$
on planar graphs is somewhat similar to the one in 
\cite{Abe} where the parameters $\chi, \ch$ and
$\chi_{DP}$ are compared.
There is quite a number of papers on adaptable choosability as well as 
separation choosability of planar graph \cite{HellZhu,CasselgrenGranholmRaspaud,ChenLihWang,
EsperetMontassierZhu,Choi, ChenLihWang}, but few where both graph coloring invariants are treated.
Dvorak et al.\ \cite{DvorakEsperetKangOzeki} noted that planar graphs have single conflict chromatic 
number at most $4$, and triangle-free at most $3$.
 Here we point out that large families of planar 
 graphs with triangles
have single conflict chromatic number $3$.

It was proved by Choi et al.\ \cite{Choi} that planar graphs 
without $4$-cycles are separation $3$-choosable,
while there are such graphs with adaptable choosability $4$ \cite{EsperetMontassierZhu}, 
so these invariants
may have different values on the family of planar graphs. 

We give an example of a multigraph $G$ with edge multiplicity
$2$, satisfying that $\single(G) = 4$,
$\ch_{ad}(G)=\ch_{sep}(G)=3$; it remains an open question whether
there exists such a planar graph without multiple edges.
Additionally, we give an example of a planar graph $G$ where
$4=\single(G) = \ch_{ad}(G) > \ch_{sep}(G) = \chi_{ad}(G) =3$.

We conclude our investigation of planar graphs by characterizing for which sequences $(a,b,c)$
there is a planar graph $G$ with $(\ch_{sep}(G), \ch_{ad}(G), \single(G))=(a,b,c)$, except for the open question
mentioned in 
the preceding paragraph, and
that it remains an open problem whether all planar graphs 
are separation $3$-choosable (which was first suggested by Skrekovski
\cite{Skrekovski}).

\section{Characterization of $2$-list assignments}

In this section we prove some results on graphs with small values of the aforementioned
chromatic invariants.
Let us start with the characterization of $2$-choosable graphs, which was proved by
Erdös et al.\ \cite{ERT}. The {\em core} of a graph $G$ is obtained by successively removing
vertices of degree $1$ from  $G$. We define the graph $\Theta_{i,j,k}$ by taking two
vertices and joining them by $3$ internally disjoint paths of lengths $i,j,k$, respectively. Such graphs are
usually referred to as {\em theta graphs}.

\begin{theorem}
\label{th:2choosable} \cite{ERT}
	A connected graph is $2$-choosable if and only if its core is isomorphic to a 
	cycle of even length or
	to $\Theta_{2,2,2k}$, for some $k \geq 1$.
\end{theorem}

Since $\ch(G) \geq \ch_{ad}(G)$, the class of adaptably $2$-choosable graphs is larger
than the family of $2$-choosable graphs. 
Indeed, the following was proved in \cite{KostochkaZhu}.

\begin{theorem}
\label{th:2adapted} \cite{KostochkaZhu}
	A connected graph $G$ with minimum degree at least $2$ is adaptably 
	$2$-choosable if and only if $G$ consists 	
	of two or three internally disjoint paths connecting two distinct vertices of $G$.
\end{theorem}

	In the following, we shall characterize graphs with single conflict chromatic
	number $2$, as well as give a partial characterization of graphs with separation choosability $2$.
	Let us also note that the graphs with DP-chromatic number $2$ are precisely the
	trees \cite{Abe}.

	We start with the case of single conflict coloring.
	By the inequality \eqref{eq:olika}, it suffices to check the graphs 
	in Theorem \ref{th:2adapted}. Straightforward case analysis
	yields that the graphs that are single conflict $2$-colorable are the same as the 
	adaptably $2$-choosable ones.
	A full proof appears in \cite{Eriksson}, and we omit it here.

\begin{proposition}
\label{prop:singconf2}
	A connected graph $G$ with minimum degree at least $2$ is single conflict 
	$2$-colorable if and only if $G$ consists 	
	of two or three internally disjoint paths connecting two distinct vertices of $G$.
\end{proposition}

	Next, we give an informal description of graphs 
        with separation choosability $2$ in the  theorem below. 
	First we need some definitions.
	Let $P = v_1 e_1 v_2 \dots e_{d-1} v_d$ be a path.
	The sequence $C = v_1 e_1 v_2 \dots v_d e_d v_1$ is an
	{\em (ordered) cycle} if $e_d = v_d v_1$.
	Similarly, the sequence $D = v_1 e_1 v_2 \dots v_d e_d v_j$
	is called an {\em (ordered) lollipop} if $e_d = v_d v_j$
	and $j \in \{2, \dots , d-2\}$.

	In \cite{Casselgren2}, the first author gave a characterization of which $2$-list assignments
	of a graph admits a list coloring of the graph.
	Indeed, using Lemma 9 in \cite{Casselgren2} it is straightforward to prove the following partial characterization of
	separation $2$-choosable graphs. Again, the full details appear in \cite{Eriksson},
	and we omit them here.

\begin{theorem}
\label{th:2sepchoice}
	If a graph $G$ is not separation $2$-choosable, then it contains
	two subgraphs $H_1$ and $H_2$ such that for $i=1,2$:
\begin{itemize}
	\item[(i)] $H_i$ is a lollipop or a cycle;

	\item[(ii)] the cycle contained in $H_i$ has length at least $4$;

	\item[(iii)] the first vertices of $H_1$ and $H_2$, respectively, are identical;

	\item[(iv)] the second vertices of $H_1$ and $H_2$ are distinct;

	\item[(v)] if $H_1$ or $H_2$ is a cycle, then the second to last vertices of
	$H_1$ and $H_2$, respectively, are distinct;

	\item[(vi)] if both $H_1$ and $H_2$ are cycles, then the second 
	and second to last vertices of $H_1$ and $H_2$,
 	respectively, are four distinct vertices.

\end{itemize}
\end{theorem}
	We remark that it is possible to give a more elaborate description 
	of separation $2$-choosability.
	Indeed, the condition that lists of adjacent vertices contain 
	at most one common color, 
	implies that $H_1$ and $H_2$ can only intersect in a rather restricted way. However, such a
	characterization would require stating a long list of rather technical conditions on 
	how the cycles in $H_1$ and
	$H_2$ may intersect, particularly in the case when these cycles have length $4$. 
	We believe that
	such a long list is not very informative, on the contrary quite technical and tedious to
	read, so we shall be content with the somewhat less precise characterization in 
	Theorem \ref{th:2sepchoice}.
	
Nevertheless, we note that a theta graph is thus separation 
$2$-choosable, 
as is also $K_4$,
and some, but not all, subdivisions of $K_4$.
$K_5$ is, however, not separation $2$-choosable, as it contains 
       an ordered lollipop and
an ordered cycle with exactly one common edge.
More generally, the following proposition is a partial converse of Theorem 
        \ref{th:2sepchoice}. The proof is a simple verification, so we omit it.

\begin{proposition}
\label{prop:disjointcycles}
If a connected graph has (at least) two
cycles of length at least $4$ with at most one common vertex,
 then it is 
not separation $2$-choosable.
\end{proposition}

	By contrast, the following follows from Theorem \ref{th:2sepchoice}.

\begin{corollary}
\label{cor:triangle}
	Every graph with at most one cycle of length greater than $3$ is separation 
	$2$-choosable.
\end{corollary}
	
Since every graph with at least two disjoint triangles is not adaptably $2$-choosable,
Corollary \ref{cor:triangle} yields an infinite
family of graphs $\mathcal{G}$ such that if $G \in \mathcal{G}$, then 
$\ch_{ad}(G) \neq \ch_{sep}(G)$.

%
%

In general, we would like to suggest the following problem, which to the best of our knowledge
has not been considered in the literature.

\begin{problem}
	For every positive integer $k \geq 0$, is there a graph $G_k$ such that
	$\ch_{ad}(G_k) - \ch_{sep}(G_k) \geq k$?
\end{problem}

As noted after Proposition \ref{prop:chsingle} in the introduction, the analogous question for single 
conflict chromatic number compared to
adaptable choosability (and thus also separation choosability) 
is known to have a positive answer as observed in 
\cite{DvorakEsperetKangOzeki}.
Indeed, although the class of adaptably $2$-choosable graphs and the one with
	single conflict chromatic number $2$ coincide, these two invariants are far from equal in general.
	Furedi et al.\ \cite{FurediKostochkaKumbhat} proved that for
	complete $k$-partite graphs $K_{k\times n}$ with $n$ vertices in each part
	$$\ch_{sep}(K_{k\times n}) =\ch(K_{k\times n}) = 
	(1+o(1))\log_{k/(k-1)} n,$$
	which also means that the adaptable choosability has the same value, since
	$\ch_{sep}(G) \leq \ch_{ad} (G) \leq \ch(G)$ for every graph $G$.
	On the other hand, Dvorak et al.\ \cite{DvorakEsperetKangOzeki} proved that
	for every graph $G$ with average degree $d$, 
	$\single(G) = \Omega(\sqrt{d/\log d})$.
	So in general $\single$ and $\ch_{ad}$ may be arbitrarily far apart, although
	it is not clear how different the families with small values of 
	these invariants are.

\bigskip

On the other hand, there are infinite families of graph with maximum degree at most $4$
where all three invariants are equal.

\begin{proposition}
\label{prop:all3}
	If $G$ is a connceted graph with $\Delta(G) \leq 4$ containing (at least)
	two disjoint cycles of length at least $4$, then
	$$\ch_{sep}(G)=\ch_{ad}(G) = \single(G)= \ch(G)=3.$$
\end{proposition}

This proposition is a consequence of Proposition \ref{prop:disjointcycles} and 
the following upper bound, first stated explicitly for adaptable choosability in \cite{KostochkaZhu}.

\begin{proposition}
\label{prop:maxdeg}
	For any graph $G$, $\single(G) \leq \lceil \Delta(G)/2 \rceil +1$.
\end{proposition}

In turn, Proposition \ref{prop:maxdeg} follows from the fact that the following observation, already made in the
original paper by Kratochvil et al.\ \cite{KratochvilTuzaVoigt} on choosability with separation, 
also holds in the setting of single conflict coloring (and thus also for adaptable choosability),
as noted in \cite{DvorakEsperetKangOzeki}.

\begin{theorem} \cite{KratochvilTuzaVoigt}
\label{th:orient}
	If there is an orientation of a graph $G$ where every vertex has outdegree
	at most $k$, then $G$ is separation $(k+1)$-choosable.
\end{theorem}

Clearly, the single conflict coloring version of Theorem \ref{th:orient} implies Proposition \ref{prop:maxdeg}, and
also that $k$-degenerate graphs have single conflict chromatic number at most $k+1$.
By contrast to Proposition \ref{prop:all3}, the following problem is open to the best of our knowledge.

\begin{problem}
\label{prob:alldiff}
	 Is there a graph $G$ that satisfies
	$\ch_{sep}(G) < \ch_{ad}(G) < \single(G)$?
\end{problem}
	
In Section 4, we shall answer this 
question in the negative for planar graphs.
Moreover, motivated by Problem \ref{prob:alldiff},
we investigate complete and complete bipartite graphs in the 
next section. 
	
\section{Complete and complete bipartite graphs}

Let us first consider complete graphs. Using the upper bound in
Proposition \ref{prop:maxdeg} and some rather straightforward
case analysis, we have obtained the results for 
complete graphs of small order in Table 1.
Some further details on how these equalities can be proved can be found in \cite{Eriksson}.
\begin{table}[H]
    \centering
        \begin{center}
        \begin{tabular}{|l|c|c|c|c|}
         \hline
         \rule{0pt}{2.5ex} $G$ & $\ch(G)$ & $\single(G)$ & $\ch_{ad}(G)$ & $\ch_{sep}(G)$ \\
         \hline
         \hline
         $K_1$& 1&1&1&1 \\ 
         $K_2$& 2&2&2&2\\
         $K_3$& 3&2&2&2\\
         $K_4$& 4&3&3&2\\
         $K_5$& 5&3&3&3\\
         \hline
        \end{tabular}
        \end{center}
    \caption{Calculated values of the invariants $\ch(K_n)$, $\single(K_n)$, $\ch_{ad}(K_n)$, and $\ch_{sep}(K_n)$ for the complete graphs $K_n$, $n=1,\dots, 5$.}
    \label{tab: Values of invariants for complete graphs}
\end{table}
	Let us also remark that in \cite{Eriksson}, it is shown that 
	$\ch_{sep}(K_n) =3$ for $n =6,7,8$, and we believe that
	the same holds for single conflict coloring and adaptable choosabilty. Thus, it appears unlikely that 
	Problem \ref{prob:alldiff} has a positive answer for complete graphs of small order.

	Let us now turn to complete bipartite graphs. We present 
	some results for $K_{1,n}, K_{2,n},K_{3,n}$, and $K_{4,n}$. 
	The proofs rely on the following proposition, which determines the threshold 
	values of $n$ at which the values of the invariants for $K_{k,n}$ increase from $k$ 
	to $k+1$.

\begin{proposition} 
\label{Lemma: Bipartite invariants k -> k+1}
\newcommand{\kkn}{K_{k,n}} 
For $n\leq k^k-1$, we have
$$
\ch(\kkn),\sing(\kkn),\chad(\kkn),\chsep(\kkn)\leq k,
$$
and for $n\geq k^k$, we have
$$
\ch(\kkn)=\sing(\kkn)=\chad(\kkn)=\chsep(\kkn)=k+1.
$$
\end{proposition}
\begin{proof}
\newcommand{\kkn}{K_{k,n}}
\newcommand{\K}{K_{k, k^k}}

    Fix $k\geq 2$ and let $V(\kkn)=\{x_1, \dots, x_k\}\cup\{y_1, \dots, y_n\}$. 

    For $n\leq k^k-1$, it suffices to prove that $\ch(\kkn)$ and $\sing(\kkn)$ are bounded by $k$. First, we show the upper bound for $\ch(\kkn)$. Let $L$ be any $k$-list assignment and consider the following two cases; either $L(x_i)\cap L(x_j)\neq \emptyset$ for some $i\neq j$, or $L(x_1), \dots, L(x_k)$ are pairwise disjoint.

    In the first case, we can assume $L(x_1)\cap L(x_2)\supseteq \{a_1\}$. Color the vertices $x_1$ and $x_2$ with $a_1$, and then $x_j$ with $a_{j-1}\in L(x_j)$, for $j=3,\dots, k$. We have now used at most $k-1$ colors on the part $\{x_1,\dots, x_k\}$, and since $|L(y_i)|>k-1$ for all $i\in [n]$, there is at least one color in $ L(y_i)\setminus \{a_1,\dots, a_{k-1}\}$ for coloring $y_i$, for $i\in [n]$.

    In the second case, there are $k^k$ ways of coloring the vertices $x_1,\dots, x_k$, and there are less than $k^k$ lists $L(y_i)$, $i\in [n]$. Thus, there must exist (at least) one choice of colors $a_1,\dots, a_k$ for $x_1,\dots, x_k$ such that $L(y_i)\neq \{a_1, \dots, a_k\}$, for all $i\in [n]$. 
Hence, there is an $L$-coloring of $K_{k,n}$

    Next, we show the upper bound $\sing(\kkn)\leq k$. Suppose that $\{c_v\}$ is a local $k$-partition of $\kkn$. There are $k^k$ ways of coloring $x_1,\dots, x_k$, each of which corresponds to a $k$-tuple $(a_1,\dots, a_k)$ of colors. There are less than $k^k$ vertices $y_i$, $i\in [n]$, each of which incident with the edges $y_ix_1, \dots, y_ix_k$, yielding a $k$-tuple of conflict colors $(c_{x_1}(y_ix_1), \dots, c_{x_k}(y_ix_k))$. As before, there must be at least one choice of colors $a_1, \dots, a_k$ for $x_1, \dots, x_k$ such that $(a_1,\dots, a_k)\neq ((c_{x_1}(y_ix_1), \dots, c_{x_k}(y_ix_k)))$ for all $i\in [n]$. Consequently, there is a single conflict coloring of $K_{k,n}$.

    For $n\geq k^k$, it suffices to prove that $\chsep(\K)>k$. Since $\kkn$ is $k$-degenerate, all four invariants are bounded above by $k+1$. We define a separated $k$-list assignment $L$ for $\K$ that does not admit an $L$-coloring as follows. Let $L(x_i)=\{c_{i,1},\dots, c_{i,k}\}$ for $i\in [k]$, where all $c_{i,j}$ are distinct colors. Let the lists $L(y_1),\dots, L(y_k)$ be all possible $k$-sets where exactly one element from each list $L(x_i)$ is contained in each $L(y_j)$, $j\in [k]$; there are $k^k$ such sets. Then, for each choice of colors for $x_1,\dots, x_k$ there is exactly one list $L(y_j)$, $j\in [k^k]$, containing exactly these colors. 
Thus there is no $L$-coloring of $\K$, so $\chsep(\K)>k$. 
\end{proof}

Let us first note the following trivial consequence of the preceding proposition and the fact that
$\chsep(K_{2,4})=3$.

\begin{itemize}

    \item[(i)] For $n\geq 1$, 
$\ch(K_{1,n})=\sing(K_{1,n})=\chad(K_{1,n})=\chsep(K_{1,n})=2,$

\item[(ii)] for $n\in \{2,3\}$,
    $\ch(K_{2,n})=\sing(K_{2,n})=\chad(K_{2,n})=\chsep(K_{2,n})=2,$
    and 

\item[(iii)] for $n\geq 4$, $\ch(K_{2,n})=\sing(K_{2,n})=\chad(K_{2,n})=\chsep(K_{2,n})=3.$

\end{itemize}

	The following corollaries
	can all be proved using Proposition \ref{Lemma: Bipartite invariants k -> k+1} and by some additional
	case analysis. The full proofs appear in \cite{Eriksson}, and we omit them here. 

\begin{corollary}
    For $3\leq n \leq 26$,
    $$
    \ch(K_{3,n})=\sing(K_{3,n})=\chad(K_{3,n})=\chsep(K_{3,n})=3,
    $$
    and for $n\geq 27$,
    $$
    \ch(K_{3,n})=\sing(K_{3,n})=\chad(K_{3,n})=\chsep(K_{3,n})=4.
    $$
\end{corollary}

\begin{corollary}
\label{Proposition: Invariants for K_4,n}
    For $4\leq n\leq 20$,
    $$
    \ch(K_{4,n})=\chad(K_{4,n})=\chsep(K_{4,n})=3,
    $$
    for $21 \leq n\leq 255$,
    $$
    \ch(K_{4,n})=\sing(K_{4,n})=\chad(K_{4,n})=\chsep(K_{4,n})=4, 
    $$
    and for $n\geq 256$,
    $$
    \ch(K_{4,n})=\sing(K_{4,n})=\chad(K_{4,n})=\chsep(K_{4,n})=5.
    $$
\end{corollary}

\begin{corollary}
    For $19\leq n\leq 255,$
    $$
    \sing(K_{4,n})=4.
    $$
\end{corollary}

As a final remark, by Proposition \ref{prop:maxdeg},
$$
\sing(K_{4,4})\leq \left \lceil\frac{4}{2}\right \rceil +1=3,
$$
so the unknown single conflict chromatic numbers for $K_{4,n}$ are 
those where $5\leq n\leq 18$.


\section{Planar graphs}

In this section, we turn our attention to the questions of which values the invariants $\chsep(G)$, $\chad(G)$, and $\sing(G)$ can attain for a planar graph $G$.
Much research has been directed towards separation and adaptable choosability of
planar graphs \cite{HellZhu,CasselgrenGranholmRaspaud,ChenLihWang,
EsperetMontassierZhu,Choi}. 
Using Theorem \ref{th:orient}, Kratochvil et al.\ \cite{KratochvilTuzaVoigt} prove that planar graphs
have separation choosability at most $4$, and triangle-free planar graphs have
separation choosability at most $3$. Since this theorem translates
both to the setting of adaptable choosability as well as single conflict coloring,
the same upper bounds hold for these invariants, as noted in 
\cite{DvorakEsperetKangOzeki}.

It has been conjectured by Skrekovski that all planar graphs are separation $3$-choosable
\cite{Skrekovski}, while there are planar graphs with adaptable choosability $4$
\cite{HellZhu}. 
In \cite{DvorakEsperetKangOzeki}, the authors remark that if Skrekovski's conjecture is true, 
then this would imply that separation choosability and adaptable choosability can be
distinct for some planar graphs. In fact,
Montassier et al.\ \cite{EsperetMontassierZhu} gave examples
of $4$-cycle-free planar graphs with adaptable choosability $4$, 
while such graphs
have been proved to be separation $3$-choosable by Choi et al.\ 
\cite{Choi}.
Guan and Zhu \cite{GuanZhu} proved that a planar graph 
is adaptably 3-choosable (and hence separation $3$-choosable) 
if any two triangles have distance at least 2 and no triangle is adjacent to a 4-cycle,
where two cycles are {\em adjacent} if they share a common
edge; they are {\em intersecting} if they share a common vertex.

Generalizing the result of Choi et al.\ \cite{Choi} and also the result by 
Guan and Zhu \cite{GuanZhu},
the following was proved by the first author along with two 
co-authors \cite{CasselgrenGranholmRaspaud}.

\begin{theorem}
\label{prop:cassel}  \cite{CasselgrenGranholmRaspaud}
	A planar graph is adaptably $3$-choosable (and hence separation $3$-choosable)
	if it satisfies any of the following conditions
\begin{itemize}

	\item[(i)] no two triangles are intersecting, and every triangle is adjacent to at most
	one $4$-cycle;

	\item[(ii)] no triangle is adjacent to a triangle or a $4$-cycle, and every $5$-cycle
	is adjacent to at most three triangles.

\end{itemize}

\end{theorem}
The proof of this proposition is based on the adaptable choosability variant
of Theorem \ref{th:orient}, so the same 
proof in fact yields large families of planar graphs with single conflict
chromatic number $3$.

\begin{corollary}
\label{cor:planarsingle}
	A planar graph satisfying (i) or (ii) of Theorem \ref{prop:cassel} is single conflict $3$-colorable.
\end{corollary}

As far as we know, this yields the only known nontrivial 
family of planar graphs with triangles and single conflict
chromatic number $3$. It would be interesting to investigate if Theorem 
\ref{prop:cassel} could be generalized further. We remark
that some further families of planar graphs with single conflict
chromatic number $3$ appear in \cite{CasselgrenGranholmRaspaud}.

More generally, one might ask which planar graphs $G$ satisfy 
$\single(G) = 3$ and if there are planar graphs with adaptable choosability $3$
and single conflict chromatic number $4$. We have been unable to find such examples,
but allowing multiple edges implies that there indeed is such graphs.

\begin{theorem} 
\label{Proposition: Multigraph planar (3,3,4)}
    There is a planar multigraph $G$ with 
    $$
    (\chsep(G), \chad(G), \sing(G))=(3,3,4).
    $$
\end{theorem}
\begin{proof}
Consider the graph $H$ in Figure \ref{fig: Multigraph planar (3,3,4)} with
a given local partition. We 
take three copies $G_1, G_2, G_3$ of the graph $H$, and denote the 
vertex in $G_i$
corresponding to $u$ by $u_i$, to $x$ by $x_i$, to $y$ by $y_i$, 
and to $z$ by $z_i$.
Next,
\begin{itemize}
    \item for $G_1$, we take $(a,b,c)=(1,2,3)$,
    \item for $G_2$, we take $(a,b,c)=(2,1,3)$, and
    \item for $G_3$, we take $(a,b,c)=(3,1,2)$.
\end{itemize}
The local $3$-partition of $G_i$ with colors $(a,b,c)$ has the property 
that all four possible combinations $(b,b),(b,c),(c,b),(c,c)$ for 
colors of $x_i$ and $z_i$ implies that there is no legal
choice of color for $y_i$, given that the vertex $u_i$ is colored $a$ 
for $i\in [3]$. Next, identify $u_1, u_2, u_3$ into a single vertex 
$u$, and denote the obtained graph by $G$. Let $\{c_v\}$ be the 
obtained local $3$-partition of $G$. Because of the aforementioned
property, $G$ 
is not conflict $\{c_v\}$-colorable, and thus $\sing(G)=4$. 
\begin{figure}
    \centering
        \begin{tikzpicture}        
        \node[mynode] (v_0) at (0,0) {};
        \node[anchor=north] at (v_0.south) {$u$};
        \node[mynode] (v_1) at (-3,4) {};
        \node[anchor=east] at (v_1.west) {$x$};
        \node[mynode] (v_2) at (0,4) {};
        \node[anchor=south] at (v_2.north) {$y$};
        \node[mynode] (v_3) at (3,4) {};
        \node[anchor=west] at (v_3.east) {$z$};
        
        \draw (v_0) --  node[myconflictx, left] {$a$}
                        node[myconflicty, left] {$a$} (v_1);
        \draw (v_0) --  node[myconflictx, left] {$a$} 
                        node[myconflicty, left] {$a$} (v_2);
        \draw (v_0) --  node[myconflictx, right] {$a$} 
                        node[myconflicty, right] {$a$} (v_3);
        \draw (v_1) to[bend left]   node[myconflictx, above, inner sep =0.5mm] {$b$}
                                    node[myconflicty, above, inner sep =0.5mm] {$b$} (v_2);
        \draw (v_1) to[bend right]  node[myconflictx, above, inner sep =0.5mm] {$c$}
                                    node[myconflicty, above, inner sep =0.5mm] {$b$} (v_2);
        \draw (v_2) to[bend left]   node[myconflictx, above, inner sep =0.5mm] {$c$}
                                    node[myconflicty, above, inner sep =0.5mm] {$c$} (v_3);
        \draw (v_2) to[bend right]  node[myconflictx, above, inner sep =0.5mm] {$c$}
                                    node[myconflicty, above, inner sep =0.5mm] {$b$} (v_3);
    \end{tikzpicture}
    \caption{A multigraph $H$ used as a subgraph in the construction of a planar multigraph $G$ with $(\chsep(G), \chad(G), \sing(G))=(3,3,4)$.}
    \label{fig: Multigraph planar (3,3,4)}
\end{figure}
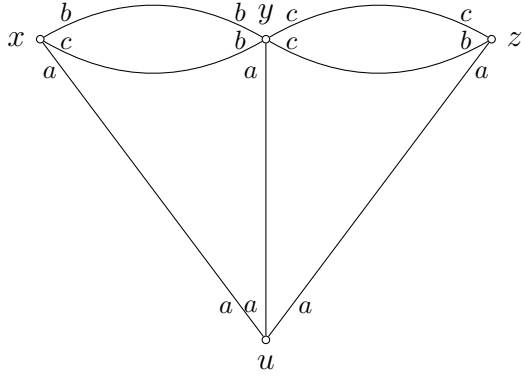

Next, we prove that $\chsep(G)=\chad(G)= 3$, by showing that the graph 
is $3$-choosable and not separation $2$-choosable. First, let $L$ be 
any $3$-list assignment and let $a\in L(u)$. We define an $L$-coloring 
$\varphi$ explicitly for the subgraphs $G_i$ of $G$. 
Set $\varphi(u)=a$, then define (in order), for each $i\in [3]$,
\begin{align*}    
\varphi(x_i)&=a_{i1}\in L(x_i)\setminus\{a\}, \\
\varphi(y_i)&=a_{i2}\in L(y_i)\setminus\{a,a_{i1}\}, \text{ and} \\\varphi(z_i)&=a_{i3}\in L(z_i)\setminus\{a, a_{i2}\}. 
\end{align*}
That $G$ is not separation $2$-choosable follows from Proposition \ref{prop:disjointcycles}.
\end{proof}

The following remains open.

\begin{problem}
    Is there a planar (simple) graph $G$ satisfying that 
    $$
    (\chsep(G), \chad(G), \sing(G))=(3,3,4)?
    $$
\end{problem}

Next, we shall prove that there is no planar graph such that
$\ch_{sep}(G) < \ch_{ad}(G) < \single(G)$, that is, Problem
\ref{prob:alldiff} has a negative answer for the family of 
planar graphs.

We shall need the following lemma, the proof of which
is left to the reader. A {\em $k$-wheel} is a graph
obtained from a cycle of length $k-1$ by adding a new vertex and
joining it to every vertex of the cycle by an edge.

\begin{lemma}
    \label{lem:wheel}
    A $6$-wheel is not separation $2$-choosable.
\end{lemma}

A {\em $\geq k$-cycle} is a cycle of length at least $k$.

\begin{theorem}
\label{prop:noplanar}
    There is no planar graph $G$ with 
    $\ch_{sep}(G) < \ch_{ad}(G) < \single(G)$.
\end{theorem}
\begin{proof}
    Assume, for a contradiction, that there is such a graph $G$. 
    Then $\ch_{sep}(G)=2$
    and $\single(G) =4$. We assume that $G$ is vertex-minimal with 
    respect to the property that $|\single(G) - \ch_{sep}(G)| \geq 2$,
    so $G$ is connected. Moreover, for every vertex $v$ of $G$, 
    $\single(G-v)=3$, and so, $\delta(G) \geq 3$.
Additionally, 
    by Proposition \ref{prop:maxdeg}, $\Delta(G) \geq 5$.

By Proposition \ref{prop:disjointcycles},
$G$ does not contain two $\geq 4$-cycles with 
at most one common vertex. 
Since $G$ has minimum degree at least $3$, it follows that it 
is $2$-connected. (Since otherwise, if $v$ is a cut-vertex, then one of 
the components of $G-v$ is outerplanar, because its longest cycle is a 
triangle.)

Let $x$ be a vertex of degree at least $5$ in $G$, and consider the set 
of neighbors $N_G(x)$ of $x$ in $G$. We shall assume that $x$ has
degree $5$ in $G$; the case when $d_G(x) > 5$ can be dealt with 
similarly.

Now, since $G$ is $2$-connected and $\delta(G-x) \geq 2$,
$G-x$ contains a cycle. Suppose first that none of the vertices
in $N_G(x)$ is contained in a cycle of $G-x$, and let $T$ be a 
vertex-minimal subtree of $G-x$ containing $N_G(x)$ (such a tree exists,
since $G-x$ is connected). Then some vertex $y$ of $N_G(x)$ is a leaf of 
$T$, and since $\delta(G) \geq 3$, this means that $y$ is a cut-vertex 
of $G$, which contradicts that $G$ is $2$-connected. 
Thus, some vertex of $N_G(x)$ is contained in a cycle in $G-x$.

Let $C$ be a cycle in $G-x$ containing the largest number of vertices
from $N_G(x)$. 
Suppose first that all vertices of $N_G(x)$ are contained in $C$.
If $C$ is a $5$-cycle, then by Lemma \ref{lem:wheel}, 
$G$ is not separation
$2$-choosable, a contradiction. On the other hand, if $C$ contains at least $6$ vertices,
then the same conclusion follows from Proposition \ref{prop:disjointcycles}.

Suppose now that four vertices from $N_G(x)$ are contained in $C$,
and that $x_5 \in N_G(x)$ is not contained in $C$.
Since $G-x$ is connected, there is a shortest path $P$ in $G-x$ from
$x_5$ to $V(C)$. If the endpoint of $P$ on $V(C)$ is not in $N_G(x)$,
or $P$ has at least three vertices, then Proposition \ref{prop:disjointcycles}
implies that $G$ is not separation $2$-choosable, so we assume
that $x_5$ is adjacent to $x_4 \in V(C) \cap N_G(x)$. 
By the choice of $C$, 
$x_5$ has no other neighbor on $C$, so since $\delta(G) \geq 3$,
$x_5$ has a neighbor $u \notin N_G(x)$. Again, by the choice
of $C$, if there is a path from $u$ to $V(C)$ in $G-x_5$, then
$x_4$ is the endpoint of such a path. However, then $G$ contains two $\geq 4$-cycles with
at most one common vertex, a contradiction to the separation $2$-choosability of $G$.
On the other hand, if there is no such path from $x_5$ to $V(C)$,
then $x_5$ is a cut-vertex of $G$, a contradiction.

Now assume at most three vertices from $N_G(x)$ are contained in $C$,
and let $x_4,x_5$ be vertices in $N_G(x)$ that are not contained in
$C$. Suppose first that $x_4$ and $C$ are contained in the same
block in $G-x$. Then there must be two internally disjoint paths from $x_4$ to $V(C))$ in $G-x$
with distinct endpoints in $C$. Since $G$ is planar,
this contradicts the choice
of $C$. Thus, $x_4$ and $C$ are contained in different blocks in $G-x$,
as is also $x_5$ and $C$. 

Consider the block graph of $G-x$. 
Now, if $x_4$ and $x_5$ are contained in the same block in $G-x$, then
either they are contained in a common cycle or, if $x_4x_5$
is a cut-edge of $G-x$, then since $\delta(G) \geq 3$, 
there is an end-block
in $G$ that contains a $\geq 4$-cycle that is disjoint from $C$. 
In both cases, $G$ contains 
two $\geq 4$-cycles with at most common vertex, 
contradicting that $G$ is
separation $2$-choosable. Hence, $x_4$, $x_5$ and $C$ are contained in
three different blocks in $G - x$. Now, since $\delta(G) \geq 3$, this 
implies that an end-block of $G-x$ contains a 
$\geq 4$-cycle with at most one
common vertex with $C$, a contradiction to the fact that $G$ is
separation $2$-choosable.
\end{proof}

\bigskip

	Finally, let us note that there are planar graphs with 
	$$3=\chi_{ad}(G) =\ch_{sep}(G) < \ch_{ad}(G) = \single(G) =4.$$

    \begin{figure}[h]
        \centering
        \begin{tikzpicture}
        
        \node[mynode] (v_0) at (0,0) {};
        \node[anchor=north] at (v_0.south) {$v$};
        \node[mynode] (v_1) at (-5,2) {};
        \node[anchor=east] at (v_1.west) {$x_1$};
        \node[mynode] (v_2) at (-3,2) {};
        \node[anchor=west] at (v_2.east) {$y_1$};
        \node[mynode] (v_3) at (-1,2) {};
        \node[anchor=east] at (v_3.west) {$x_2$};
        \node[mynode] (v_4) at (1,2) {};
        \node[anchor=west] at (v_4.east) {$y_2$};
        \node[mynode] (v_5) at (3,2) {};
        \node[anchor=east] at (v_5.west) {$x_3$};
        \node[mynode] (v_6) at (5,2) {};
        \node[anchor=west] at (v_6.east) {$y_3$};
        \node[mynode] (v_7) at (0,4) {};
        \node[anchor=south] at (v_7.north) {$u$};
        
        \draw (v_0) -- node[myedge, below] {$c$} (v_1);
        \draw (v_0) -- node[myedge, below] {$c$} (v_2);
        \draw (v_0) -- node[myedge, below] {$b$} (v_3);
        \draw (v_0) -- node[myedge, below] {$b$} (v_4);
        \draw (v_0) -- node[myedge, below] {$4$} (v_5);
        \draw (v_0) -- node[myedge, below] {$4$} (v_6);
        
        \draw (v_7) -- node[myedge, above] {$a$} (v_1);
        \draw (v_7) -- node[myedge, above] {$a$} (v_2);
        \draw (v_7) -- node[myedge, above] {$a$} (v_3);
        \draw (v_7) -- node[myedge, above] {$a$} (v_4);
        \draw (v_7) -- node[myedge, above] {$a$} (v_5);
        \draw (v_7) -- node[myedge, above] {$a$} (v_6);
    
        \draw (v_1) -- node[myedge, above] {$b$} (v_2);
        \draw (v_3) -- node[myedge, above] {$c$} (v_4);
        \draw (v_5) -- node[myedge, above] {$b$} (v_6);
     \end{tikzpicture}
\caption{A planar graph with a $4$-edge coloring.}
\label{fig:figure2}
\end{figure}
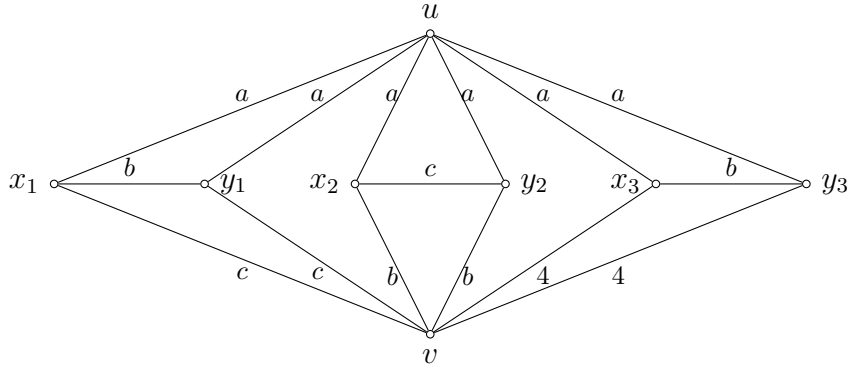

\begin{theorem}
\label{prop:adapted}
	There is a planar graph $G$ with 
    $$3=\chi_{ad}(G)=\ch_{sep}(G) < \ch_{ad}(G) = \single(G) =4.$$
\end{theorem}

\begin{proof}
	Consider the graph $G$ in Figure \ref{fig:figure2}.
	We take $3$ copies $G_1, G_2, G_3$ of the graph $G$, and denote 
    the vertices corresponding to $u$, by $u_1, u_2, u_3$, respectively, and similarly for $v$.
\begin{itemize}
	
	\item For $G_1$, we take $(a,b,c)=(1,2,3)$,

	\item for $G_2$, we take $(a,b,c) = (2,3,1)$, and

	\item for $G_3$, we take $(a,b,c) = (3,1,2)$.

\end{itemize}
	Next, we identify the vertices $u_1, u_2, u_3$ into a single vertex $u$, and denote the obtained
	graph by $H$. We define a list assignment $L$ of $H$ by letting $j \in L(x)$ if and only if
	an edge incident to $x$ is colored $j$. This yields as $3$-list assignment of $H$ for which 
	there is no adapted list coloring.
	Thus $\single(H)= \ch_{ad}(H)=4$.
	However, $H$ is adaptably $3$-colorable, since it is $3$-colorable.

    Let us show that $\chsep(H)=3$. By Proposition \ref{prop:disjointcycles}, 
    $\chsep(H)\geq 3$, so we must show $\chsep(H)\leq3$. 
    Consider a separated $3$-list assignment $L$ for $H$. We define 
    an $L$-coloring $\varphi$ explicitly for one of the subgraphs 
    $G_r$ of $H$. Set $\varphi(u_r)=j\in L(u_r)$ and 
$\varphi(v_r)=k\in L(v_r)$. For every $i\in [3]$, color $x_i$ and $y_i$ according to the 
    following two cases (with vertex labels as in Figure \ref{fig:figure2}).

    \vspace{\baselineskip}

    \textbf{Case 1.} $\{j,k\}\subseteq L(x_i)$. \hfill \\
    Define $\varphi(x_i)=c_i\in L(x_i)\setminus\{j,k\}$. Then, since $L(y_i)$ can only contain one of $j$, $k$, and $c_i$ it is possible to choose $\varphi(y_i)\in L(y_i)\setminus\{j,k,c_i\}$. 

    \vspace{\baselineskip}
    
    \textbf{Case 2.} $\{j,k\}\nsubseteq L(x_i)$. \hfill \\
    We assume, without loss of generality, that $k\notin L(x_i)$. Let $c_i\in L(x_i)\setminus\{j\}$ and consider the following two subcases. If $L(y_i)=\{j,k,c_i\}$, then color $x_i$ with $\varphi(x_i)\in L(x_i)\setminus \{j,c_i\}$ and let $\varphi(y_i)=c_i$. Otherwise, if $L(y_i)\neq\{j,k,c_i\}$, define $\varphi(y_i)\in L(y_i)\setminus\{j,k,c_i\}$ and set $\varphi(x_i)=c_i$. 

    \vspace{\baselineskip}
    
    Color the other subgraphs $G_r$ of $H$ in the same way, keeping the color of $u_r$. 
The obtained coloring $\varphi$ is an $L$-coloring. Hence, $\chsep(H)\leq3$.
\end{proof}

As regards the relation between $\ch_{sep}$ and $\chi_{ad}$,  let us briefly remark that
all bipartite graphs have adaptable chromatic number $2$, while e.g.\ planar graphs,
such as $K_{2,4}$, satisfy
$\ch_{sep}(K_{2,4}) = 3$. Conversely, there are known examples of 
planar graphs with $\chi_{ad} > \ch_{sep}$.
Indeed, Chen et al.\ \cite{ChenLihWang} proved that planar graphs without
$5$-cycles and normally adjacent $4$-cycles are separation $3$-choosable.
Nevertheless, Esperet et al.\ \cite{EsperetMontassierZhu} showed that there are $5$-cycle-free
planar graphs $G$ with adaptable chromatic number $4$. Moreover, these examples satisfy that all $4$-cycles
are pairwise disjoint, so they are separation $3$-choosable.

\bigskip

	We conclude this section by deducing the following, which resembles
	the main result of \cite{Abe}. Note that we cannot prove a full characterization
	in the following corollary; in particular, it remains 
	an open problem whether  all planar graphs are separation $3$-choosable.

\begin{corollary}
\label{cor:planar}
	For $(i,j,k) \in \mathbf{N}^3 \setminus \{(3,3,4),(4,4,4)\}$,
	there is a planar graph $G$ with 
	$(\ch_{sep}(G), \ch_{ad}(G), \single(G)) = (i,j,k)$ 
	if and only if
	$$(i,j,k) \in \{(1,1,1), (2,2,2), (2,3,3), (3,3,3), (3,4,4)\}.$$
\end{corollary}

\begin{proof}
The case when $\ch_{sep}=1$ is trivial.
A tree satisfies that $\ch_{sep} = \ch_{ad} = \single =2$, since it is
$1$-degenerate. Since $\ch_{ad}(G) = \single(G)$ if $\ch_{ad}(G)=2$,
there is no graph with $(\ch_{sep}, \ch_{ad}, \single) =(2,2,3)$.
On the other hand, a cactus graph with at least two disjoint cycles,
and where all cycles are triangles,
satisfies $(\ch_{sep}, \ch_{ad}, \single)=(2,3,3)$.

Proposition \ref{prop:all3} yields a family of planar graphs with
 $(\ch_{sep}, \ch_{ad}, \single)=(3,3,3)$, 
and
finally, the graph in Theorem \ref{prop:adapted} shows
that there are graphs with $(\ch_{sep}, \ch_{ad}, \single)=(3,4,4)$.
\end{proof}


{}

\end{document}